\documentclass{amsart}

\usepackage{amsmath,amssymb,amsthm}
\usepackage[ruled,vlined]{algorithm2e}
\usepackage{graphicx,tikz}

\usepackage[normalem]{ulem} 

\usepackage{subcaption} 

\newtheorem{theorem}{Theorem}
\newtheorem*{thm}{Theorem}
\newtheorem*{proposition}{Proposition}
\newtheorem{corollary}{Corollary}

\newtheorem{lemma}{Lemma}

\theoremstyle{definition}
\newtheorem{definition}{Definition}

\theoremstyle{remark}

\begin{document}

\title[]{Extreme Values of the Fiedler Vector on Trees}

\author[]{Roy R. Lederman}
\address{Department of Statistics, Yale University, New Haven, CT 06511, USA}
\email{roy.lederman@yale.edu}

\author[]{Stefan Steinerberger}
\address{Department of Mathematics, University of Washington, Seattle, WA 98195, USA}
\email{steinerb@uw.edu}

\begin{abstract} Let $G$ be a tree on $n$ vertices and let $L = D-A$ denote the Laplacian matrix on $G$. The second-smallest eigenvalue $\lambda_{2}(G) > 0$, also known as
the algebraic connectivity, as well as the associated eigenvector $\phi_2$ have been of substantial interest. We investigate the question of when the maxima and minima of $\phi_2$ are
assumed at the endpoints of the longest path in $G$. Our results also apply to more general graphs that `behave globally' like a tree but can exhibit more complicated local structure. The crucial new ingredient is a reproducing formula for the eigenvector $\phi_k$.
\end{abstract}

\maketitle

\section{Introduction}
\subsection{Introduction.} 
Let $G=(V,E)$ be a simple, undirected, unweighted, connected graph on $n$ vertices $V=\left\{v_1, \dots, v_n\right\}$. 
The adjacency matrix $A \in \left\{0,1\right\}^{n \times n}$  encodes the connections between the
vertices via $A_{ij}=1$ if the $i$ and $j$ are connected by an edge $(i,j) \in E$ and $A_{ij}=0$ otherwise. 
The degree matrix $D$ is the diagonal matrix $d_{ii} = \deg(v_i) = \sum_j A_{ij}$. The Laplacian matrix of the graph is defined as
\begin{equation}\label{eq:lap}
 L = D - A.
\end{equation}
The Laplacian is symmetric and its quadratic form representation immediately implies that all its eigenvalues are nonnegative: if
$f:V \rightarrow \mathbb{R}$, then
\begin{equation}\label{eq:quad}
 \left\langle f, Lf \right\rangle = \sum_{u \sim_E v} (f(u) - f(v))^2.
 \end{equation}
We order its eigenvalues by their size
$$ \lambda_n(G) \geq \lambda_{n-1}(G) \geq \dots \geq \lambda_{2}(G) \geq \lambda_1(G) = 0.$$
We refer to \cite{chung, davies, mohar} for an introduction to spectral graph theory.
It is not difficult to see that the unique eigenvector associated with the eigenvalue $\lambda_1(G)=0$ is the vector having constant entries and that 
\begin{equation} \label{one}
 \lambda_{2}(G) = \min_{x \perp \textbf{1}} \frac{\sum_{v_i \sim_E v_j}{(x_i - x_j)^2}}{\sum_{i=1}^{n}{x_i^2}}.
 \end{equation}
Since the graph is assumed to be connected, it follows that $\lambda_{2}(G) > 0$ if and only if $G$ is connected.
Any eigenvector $\phi_2$ associated with the second smallest eigenvalue is
also known as the Fiedler vector \cite{fiedler1, fiedler2, fiedler3, gern, stone}.  The following seminal result for a graph $G$ and its eigenvector $\phi_2$ is due to Fiedler \cite{fiedler3}.

\begin{thm}[Fiedler] Let $G=(V,E)$ be a simple, undirected, unweighted, connected graph on $n$ vertices $V=\left\{v_1, \dots, v_n\right\}$. The induced subgraph of $G = (V,E)$ on $\left\{v \in V: \phi_2(v) \geq 0\right\}$ is connected.
\end{thm}
An induced subgraph $G' = (V',E')$  here means the graph with the nodes $V' = \left\{v \in V: \phi_2(v) \geq 0\right\}$, and the set of edges connecting vertices in this subset, i.e. $E' = \left\{ (i,j) \in E : i,j \in V' \right\}$.
Since $-\phi_2$ is also an eigenvector with eigenvalue $\lambda_2$, the same theorem implies that $\left\{v \in V: \phi_2(v) \leq 0\right\}$ is also connected. We assume that all eigenvectors mentioned in this paper are normalized $\|\phi_k\|_2 =1$.\\

Fiedler's theorem, together with many other desirable properties, motivates the classical spectral cut whereby the sign of $\phi_2$ is used to decompose a graph. Overall, relatively little seems to be known about the actual behavior of the Fiedler vector:
\begin{quote}
However, apart from the original results from M. Fiedler,  very few is known about the
Fiedler vector and its connection to topological properties of the underlying graph [...] (\cite{gern}, 2018)
\end{quote}

Simultaneously, these types of questions have become increasingly important in the framework of Graph Signal Processing we refer to \cite{ham, irion, or, per, saito, shu, shu2, shu3} for an introduction into recent work. Many studies of the properties of graphs use the properties that are easier to demonstrate on special families of graphs, such as {\em paths} and {\em trees}; we present a general framework which is particularly effective on trees but is also useful in a more general setting.

\subsection{The Problem} Let $G=(V,E)$ be a tree. It is compelling to interpret Equation (\ref{one}) as suggesting that $\phi_2$ is the `smoothest' vector that is orthogonal to the constants, and to further infer that the node(s) that attain the maximum (positive) value on $\phi_2$ are the furthest away from the node(s) that attain the minimum (negative) value on $\phi_2$.
This was explicitly conjectured in \cite{moo}, a counterexample was then produced by Evans \cite{evans} and is shown in Fig. 1 where maximum and minimum
are attained far away from one another but not at the two points of maximum distance from one another. A natural question that remains is to understand (1)
the behavior of extrema of the Fiedler vector and, as discussed by Lef\`evre \cite{lef} and Gernandt \& Pade \cite{gern}, (2) under which conditions such a
result might still be true.

\begin{figure}[h!]
\begin{center}
\begin{tikzpicture}
\filldraw (0,0) circle (0.07cm);
\filldraw (1,0) circle (0.07cm);
\filldraw (2,0) circle (0.07cm);
\filldraw (3,0) circle (0.07cm);
\filldraw (4,0) circle (0.07cm);
\filldraw (5,0) circle (0.07cm);
\filldraw (6,0) circle (0.07cm);
\filldraw (6,0) circle (0.07cm);
\filldraw (7,0) circle (0.07cm);
\filldraw (8,0) circle (0.07cm);
\filldraw (9,0) circle (0.07cm);
\draw [thick] (0,0) -- (9,0);
\filldraw (3,-1) circle (0.07cm);
\filldraw (2,-2) circle (0.07cm);
\filldraw (2.2,-2) circle (0.07cm);
\filldraw (2.4,-2) circle (0.07cm);
\filldraw (2.6,-2) circle (0.07cm);
\filldraw (2.8,-2) circle (0.07cm);
\filldraw (3,-2) circle (0.07cm);
\filldraw (3.2,-2) circle (0.07cm);
\filldraw (3.4,-2) circle (0.07cm);
\filldraw (3.6,-2) circle (0.07cm);
\filldraw (3.8,-2) circle (0.07cm);
\filldraw (4,-2) circle (0.07cm);
\draw [thick] (3,0) -- (3,-1);
\draw [thick] (2,-2) -- (3,-1);
\draw [thick] (2.2,-2) -- (3,-1);
\draw [thick] (2.4,-2) -- (3,-1);
\draw [thick] (2.6,-2) -- (3,-1);
\draw [thick] (2.8,-2) -- (3,-1);
\draw [thick] (3,-2) -- (3,-1);
\draw [thick] (3.2,-2) -- (3,-1);
\draw [thick] (3.4,-2) -- (3,-1);
\draw [thick] (3.6,-2) -- (3,-1);
\draw [thick] (3.8,-2) -- (3,-1);
\draw [thick] (4,-2) -- (3,-1);
\node at (3, -2.5) {minimum attained here};
\node at (9, -1) {maximum attained here};
\draw [thick, ->] (9, -0.8) -- (9, -0.3);
\end{tikzpicture}
\end{center}
\caption{The `Fiedler rose' counterexample of Evans \cite{evans}.}
\end{figure}
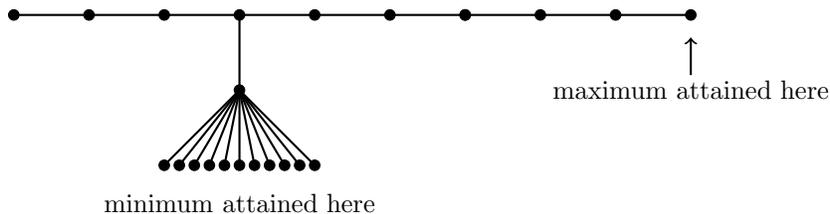

The spectral geometry of trees has attracted a lot of attention over the years: path monotonicity properties of the smallest eigenvector of the Graph Laplacian
on trees have been studied since Fiedler \cite{fiedler2}, we also refer to work of Merris \cite{merris}.  Kirkland, Neumann \& Shader \cite{kirkland} extend these results to
weighted tree (compare their Theorem 6 with our Corollary 1). Kirkland \& Neumann \cite{kirk2} discuss the effect of local graph operations on trees on the global spectrum.
Bapat, Kirkland \& Pati \cite{bapat} extend classical monotonicity results to the perturbed Laplacian (see also Kirkland \& Fallat \cite{kirk0} and Rocha \& Trevisan \cite{rocha}).
A global spectral criterion characterizing trees was given by Band \cite{band}.
We also refer to a survey of Merris \cite{merris2} and references therein.  Our paper is specifically concerned with the question of maxima and minima and is thus closer
to \cite{moo, evans, gern, lef} and our approach is more inspired by classical potential theory than algebra. We hope that our approach can also be helpful in other situations and also raise a number of questions.

\subsection{A Comment on the  Continuous Case.}
Let $\Omega \subset \mathbb{R}^2$ be a domain and consider the second smallest eigenfunction of the Laplace operator $-\Delta$ with Neumann boundary conditions, i.e. the equation
\begin{align*}
-\Delta \phi_2 &= \mu_2 \phi_2 \quad \mbox{in}~ \Omega\\
\frac{\partial \phi_2}{\partial \nu} &= 0 \quad \mbox{on}~\partial \Omega. 
\end{align*}
Rauch conjectured in 1974 that the maximum and the minimum are assumed at the boundary. This is known to fail at this level of generality \cite{b25, b4} but widely assumed to be true for convex domains $\Omega$. 
The second author \cite{stein} showed that if $x_1, x_2 \in \Omega$ satisfy 
$\|x_1 - x_2\| = \mbox{diam}(\Omega)$, then every maximum and minimum is
assumed within distance $c\cdot \mbox{inrad}(\Omega)$ of $x_1$ and $x_2$, where $\mbox{inrad}(\Omega)$ is the inradius of $\Omega$ and $c$ is a universal constant (which is the optimal scaling up to the value of $c$). Therefore, up to an inradius, the maximum and minimum are essentially assumed at maximum distance. There is no formulation of the Hot Spots conjecture on graphs (perhaps not surprising since there is no clear definition of what the boundary of a graph would be; nonetheless, similar types of phenomena do appear, see e.g. \cite{steini2} and it would be interesting to understand them better).

\section{Main Results}

We present two main results: the first is a representation formula for an eigenvector $\phi_k$ on a general graph $G$ that we find very useful in the investigation of the Fiedler vector $\phi_2$ of trees. It allows to quickly recover some of the existing results and gives a better understanding of the behavior of $\phi_2$.
In particular, it implies that for generic trees there is little to no reason to assume that the extrema of $\phi_2$ are assumed at vertices that are at distance $\mbox{diam}(G)$. However, the second contribution is an explicit application of the representation formula to construct families of graphs on which the desired statement indeed holds: the extrema of the Fiedler vector are assumed at vertices which are distance $\mbox{diam}(G)$ apart. We hope that the representation formula will also be useful in other settings.

\subsection{A Representation Formula}\label{sec:game}
Let us fix $G=(V,E)$ to be a simple, undirected, connected graph on $n$ vertices.
Let $v_{s}, v_{t}$ be two arbitrary vertices. In Algorithm \ref{alg:game} we introduce a game that results in a representation formula for any eigenvector $\phi_k$ associated with the eigenvalue $\lambda_k$.

\begin{theorem}\label{thm:1} Let  $G=(V,E)$ to be a finite, simple, undirected, connected graph. 
The expected payoff of the game in Algorithm \ref{alg:game} satisfies 
$$ \mathbb{E}\left(\emph{payoff}(v_s  \rightarrow v_t)\right) = \phi_k(v_s) - \phi_k(v_t).$$
\end{theorem}
The proof is presented in \S \ref{sec:pf:1}. 
An analogous result for the random walk normalized Laplacian $A^{}D^{-1}$ is easy to obtain (by its very nature, $AD^{-1}$ is strongly tied to random walks). Our theorem and game apply to the Kirchhoff Laplacian $L = D-A$. The game could also be interpreted as a discretized version of the Feynman-Kac formula (see e.g. \cite{taylor}).\\

\begin{algorithm}[H]\label{alg:game}
\SetAlgoLined
\KwIn{Graph $G=(V,E)$, arbitrary vertices $v_s, v_t \in V$, \\ 
eigenvalue $\lambda_k$ and a corresponding normalized eigenvector $\phi_k$ such that $(D-A)\phi_k = \lambda_k \phi_k$. }
\KwResult{payoff }
 payoff $\gets$ 0 \;
 $w \gets v_s$ \; 
\While{$w \neq v_t$}{
  payoff $\gets$ payoff $+ \lambda_{k}\cdot\phi_k(w)/\mbox{deg}(w)$\;
  neighbors($w$)  $\gets \{v \in V : (w,v) \in E \} $ \;
  $w \gets $ choose uniformly at random from neighbors($w$);
 }
 \caption{``The Game'': informally, the game is a random walk through the graph, starting at vertex $v_s$ and terminating when the vertex $v_t$ is reached. A ``payoff'' of $\lambda_{k}\cdot\phi_k(w)/\mbox{deg}(w)$ is accumulated when we visit a vertex $w$.}
\end{algorithm}
\vspace{10pt}

We believe that Theorem \ref{thm:1} has a substantial amount of explanatory power. A simple example is the following (well-known) corollary {which provides} a simple form of monotonicity of the second eigenvector along paths in a tree.  We emphasize that this type of result is not new and refer to  \cite{fiedler2, gern, lef, kirkland, merris} and references therein; however, it is well suited to illustrate an application of the main idea.

\begin{corollary}[see also \cite{fiedler2, gern, lef, kirkland, merris}] Let $G=(V,E)$ be a  simple, undirected tree.
Let $\Gamma$ be a path in the tree such that $\phi_2$ only assumes positive values on the path. If $v \neq w \in V$ are two vertices on the path and $v$ is at a greater distance than $w$ from the closest vertex where $\phi_2$ is negative, then
$$ \phi_2(v) > \phi_2(w).$$
In particular,  maxima and minima are attained in vertices with degree 1.
\end{corollary}
\begin{proof} The proof is a relatively short application of Theorem 1 and illustrates it well which is why we discuss it here. Let $z$ be the vertex where $\phi_2$ is negative or zero but where $\phi_2$ is positive for one of the neighbors. By Fiedler's theorem (with change of signs) and because trees have no cycles, that vertex is unique. Let us now assume $w$ and $v$ are vertices on a path and $v$ is from a greater distance from $z$ than $w$.

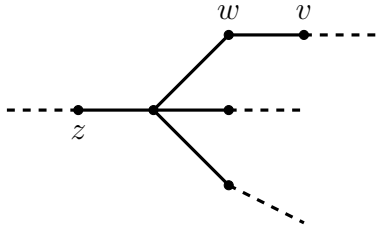
\begin{figure}[h!]
\begin{center}
\begin{tikzpicture}[scale=1]
\filldraw (0,0) circle (0.06cm);
\filldraw (1,0) circle (0.06cm);
\filldraw (2,1) circle (0.06cm);
\filldraw (2,0) circle (0.06cm);
\filldraw (2,-1) circle (0.06cm);
\filldraw (3,1) circle (0.06cm);
\draw [very thick] (0,0) -- (1,0);
\draw [very thick] (1,0) -- (2,-1);
\draw [very thick] (1,0) -- (2,0);
\draw [very thick, dashed] (2,-1) -- (3,-1.5);
\draw [very thick] (1,0) -- (2,1);
\draw [very thick] (2,1) -- (3,1);
\draw [very thick, dashed] (0,0) -- (-1,0);
\draw [very thick, dashed] (2,0) -- (3,0);
\draw [very thick, dashed] (3,1) -- (4,1);
\node at (0, -0.3) {\Large $z$};
\node at (2, 1.3) {\Large $w$};
\node at (3, 1.3) {\Large $v$};
\end{tikzpicture}
\end{center}
\label{fig:game1}
\caption{Sketch: proof of Corollary 1.}
\end{figure}

We consider an instance of the game starting at $v_s=v$ and terminating at $v_t=z$ assuming a local geometry as in Figure 2.
By Theorem \ref{thm:1}, the expected value of this game is 
\begin{equation}\mathbb{E}\left(\mbox{payoff}(v \rightarrow z)\right) = \phi_k(v) - \phi_k(z).
\end{equation}
Suppose that the sequence of traversed vertices is $$ v=a_0,a_1,a_2,...,a_{j-1},a_j,a_{j+1},...,a_{m},z$$ (where the number of steps $m+1$ is the number of steps in this particular instance). We note that the sequence of vertices must include $w$ which lies on the only path from $v$ to $z$ since $G$ is a tree ($w$ and $v$ may appear in the sequence multiple times). 
We denote by $j$ the first time the game visits $w$ such that $a_j = w$.
By construction, the Fiedler eigenvector is non-negative at every step on the tree $\phi_2(a_i) \geq 0$ before the game is terminated upon arriving to the first negative node $z$. Therefore, at each step of the sequence we accumulate a non-negative payoff $\lambda_{k}\cdot\phi_k(w)/\mbox{deg}(w) \geq 0$.
It follows that 
\begin{equation}\mathbb{E}\left(\mbox{payoff}(v \rightarrow z)\right) = \phi_k(v) - \phi_k(z) > 0.
\end{equation}

What happens if we truncate the instance of the game when we reach $w$? We obtain the truncated sequence $v=a_0,a_1,a_2,...,a_{j-1},a_j=w$. 
The truncated sequences are instances of the ``truncated'' game starting at $v_s=v$ and terminating at $v_t=w$ (and by construction, their distribution is the same as that of sequences from the truncated game). Therefore, by Theorem \ref{thm:1}, the expected payoff is:
\begin{equation}\mathbb{E}\left(\mbox{payoff}(v \rightarrow w)\right) = \phi_k(v) - \phi_k(w).
\end{equation}
The truncated sequence is strictly shorter than the original sequence and therefore collects at most the same payoff. Therefore:
\begin{equation}\phi_k(v) - \phi_k(z) = \mathbb{E}\left(\mbox{payoff}(v \rightarrow w)\right) \geq  \mathbb{E}\left(\mbox{payoff}(v \rightarrow w)\right) > 0.
\end{equation}
\end{proof}

\subsection{Counterexamples to the ``longest path hypothesis''.} We now return to the original question from \cite{moo}: whether the second eigenvector assumes maximum and minimum at the endpoints of the longest path in the tree.
This was disproven by Evans \cite{evans} by means of an explicit example. The question has also been studied in \cite{gern, lef}. The purpose of this subsection is to argue that the representation formula from Theorem 1 allows to heuristically explain why, generally, there is no reason the second eigenvector should assume extreme values at the endpoints of the longest path -- it shows that Evans' counterexample is actually representative of one of the main driving forces behind localization of large values of the Fiedler vector in sub-structures. We are not making any precise claims at this point in the paper, however, this section tries to provide a good working heuristic that (a) allows to construct counterexamples quite easily and (b) will underlie all our formal arguments later in the paper. It also seems that these ideas could be made precise in more than one way.

One of the crucial ingredients in the representation formula is the number of steps a typical random walk will need to reach another vertex: if the tree has a complicated structure (say, many vertices with large degree), then it will take a very long time for a random walk to reach a specific vertex (this principle is already embodied in Evans' counterexample \cite{evans} shown in Figure 1): distance is not as crucial as complexity -- this immediately implies a large family of counterexamples whose type is shown in Figure \ref{fig:1over4}: we consider a graph composed of a path of length $d$ attached to a tree $T$ at vertex $d/4$. We assume the tree $T$ has diameter much smaller than $d/4$ but has vertices of very large degree.

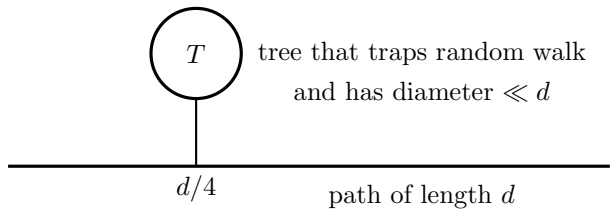
\begin{figure}[h!]
\begin{center}
\begin{tikzpicture}
\draw [very thick] (0,0) -- (8,0);
\draw[very thick] (2.5,1.5) circle (0.6cm);
\draw [thick] (2.5,0) -- (2.5,0.9);
\node at (2.5, 1.5) {$T$};
\node at (5.5, -0.4) {path of length $d$};
\node at (2.5, -0.3) {$ d/4$};
\node at (5.5, 1.5) {tree that traps random walk};
\node at (5.5, 1) {and has diameter $\ll d$};
\end{tikzpicture}
\end{center}
\caption{A generic counterexample.\label{fig:1over4}}
\end{figure}

The game then suggests that, if the tree has vertices of sufficiently large degree, one extremum is at the `most remote' part off the path in the tree $T$ -- in particular, one of the two extrema would not be on the path and thus not at the endpoints of the longest path in the graph. The distance between the extrema would be
$$ d(v_{\max}, v_{\min}) \leq \frac{3d}{4} + \mbox{diam}(T) < d.$$
A sketch of the argument to show that this type of construction works is as follows. There are two cases: either the sign change of the second eigenvector happens inside $T$ or it happens on the path. If the diameter of $T$ is sufficiently small compared to $d$, known inequalities on the eigenvalue $\lambda_2$ (which we use below) suggest that the first case cannot occur. This means that the sign change happens on the path. If the value of the eigenvector in the vertex $v$ that connects to $T$ is nonzero, we can play the game with vertices starting in $T$ and ending in $v$. Corollary 1 shows that the values of the eigenfunction inside $T$ are (in absolute value) at least as big as the value in $v$. Then the game leads to a nonzero contribution for each step of the random walk that is not arbitrarily small. This means that in order to ensure large (absolute) values inside the tree, the quantity to maximize is the expected number of steps in the game -- this, in turn, can be achieved by having vertices of large degree. We emphasize that this heuristic is non-rigorous but quickly motivates the construction of many counterexamples. All the positive results in our paper can be understood as ensuring the absence of such a structure.\\

\begin{figure}[h!]
\begin{subfigure}[b]{0.7\textwidth}
\begin{center}
\begin{tikzpicture}[scale=0.8]
\draw [very thick] (0,0) -- (8,0);
\draw [thick] (4,0) -- (4,2);
\node at (6.5, -0.4) {path of length $d$};
\node at (4, -0.4) {$d/2$};
\node at (5.5, 1) {height $\sim d/3$};
\end{tikzpicture}
\end{center}
\caption{Baseline graph: extrema are assumed at the end of the long path (not a counterexample).\label{fig:1over2a}}
\end{subfigure}
\begin{subfigure}[b]{0.7\textwidth}
\begin{center}
\begin{tikzpicture}[scale=0.8]
\draw [very thick] (0,0) -- (8,0);
\draw [thick] (4,0) -- (4,2);
\node at (6.5, -0.4) {path of length $d$};
\node at (4, -0.4) {$d/2$};
\node at (5.5, 1) {height $\sim d/3$};
\draw [thick] (4,1) -- (4.2, 1.2);
\draw [thick] (4,1) -- (4-0.2, 1.2);
\draw [thick] (4,1) -- (4-0.2, 0.8);
\draw [thick] (4,1.5) -- (4.2, 1.7);
\draw [thick] (4,1.5) -- (4-0.2, 1.3);
\draw [thick] (4,0.7) -- (4.2, 0.7);
\draw [thick] (4,0.7) -- (3.8, 0.5);
\draw [thick] (4,0.7) -- (3.8, 0.7);
\end{tikzpicture}
\end{center}
\caption{Adding little trees to the path in the middle leads to a counterexample.\label{fig:1over2b}}
\end{subfigure}
\caption{Constructing another type of counterexample.\label{fig:1over2}}
\end{figure}

To build further intuition, we quickly sketch another type of counterexample in Figure \ref{fig:1over2}. Take a path graph of length $d$ and
add a path graph of length $d/3$ to the middle vertex as in Figure \ref{fig:1over2a}. What we observe is that the eigenvector changes along the
long path, that it assumes extrema at its end and that the eigenvector is small and changes slowly on the little path in the middle.
However, if we start adding paths of length 1 to the vertices of the short path in the middle (or short trees, even ones with bounded
diameter) as in Figure  \ref{fig:1over2b}, then after a while the eigenvector flips and assumes an extremum at the tip of short path in the middle.
Perhaps the main contribution of our paper is a framework that clearly establishes \textit{why} this happens. The
theorems we give are one way of capturing the phenomenon but presumably there are many other possible formulations
that could be proven by formalizing the same kind of mechanism that we use here.
A particular consequence of these ideas is that a \textit{typical} tree (say, chosen with respect to the uniform measure on the set of all connected trees of a fixed size) should not have the desired property of $\phi_2$ assuming its extrema
at the endpoints of a path of length $\mbox{diam}(G)$. This appears to be an interesting problem. We refer to numerical work done by Lef\`evre \cite{lef} showing that all trees
with $n \leq 11$ vertices do have the property but already $2\%$ of trees with $n=20$ vertices do not. Lef\`evre specifically asks
whether a typical tree on $n$ vertices does not have the property as $n$ becomes large and we also consider this to be an interesting
problem.

\subsection{An Admissible Class.}
The purpose of this section is to construct a large family of tree-like graphs for which the following statement is true: the second eigenvector of the graph Laplacian does indeed assume maximum and minimum at the endpoints of the longest path (which, for graphs of this type, will be unique). 

We begin our construction with a class of graphs to which we will refer here as {\em augmented path graphs}; we will denote the class by $\mathcal{P}$. Augmented path graphs are paths with subgraphs attached to them as illustrated in Figure \ref{fig:bubble}.

\begin{definition}[augmented path graph $\mathcal{P}$]\label{def:path-based-graph}
An augmented path graphs $G=(V,E) \in \mathcal{P}$ of diameter $d=\mbox{diam}(G)$ is a simple, undirected, unweighted, connected graph of diameter $d$ that can be written as the union the following subgraphs:
\begin{enumerate}
\item A path of length $d$, denoted by $G_p = (V_p, E_p)$, with vertices $\{V_p(1),...,V_p(d)\}$.
\item For each $1 \leq k \leq d$, zero or more subgraphs denoted $G_{k,i}=(V_{k,i},E_{k,i})$, each of them connected and has at least two vertices $V_{k,i} = \{V_{k,i}(0),V_{k,i}(1),...\}$, such that $V_{k,i}(0)=V_p(k)$ is a vertex on the path $G_p$ and $G_{k,i}\backslash V_{k,i}(0)$ is also a connected subgraph.
\end{enumerate}
Furthermore, $V_{k,i}\backslash V_{k,i}(0)$ and $V_{k',i'}\backslash V_{k',i'}(0)$  are disjoint sets (unless $k=k'$ and $i=i'$) with no edges in $E$ connecting them.
\end{definition}

In other words, the subgraphs $G_{k,i}$ are connected only through vertices on the path $G_p$. 
The only requirement at this point is that the attached graphs do not increase the overall diameter of the graph beyond $d=\mbox{diam}(G)$.

\begin{figure}[h!]
\begin{center}
\begin{tikzpicture}[scale=1]
\filldraw (0,0) circle (0.06cm);
\filldraw (1,0) circle (0.06cm);
\filldraw (2,0) circle (0.06cm);
\filldraw (3,0) circle (0.06cm);
\filldraw (4,0) circle (0.06cm);
\filldraw (5,0) circle (0.06cm);
\filldraw (6,0) circle (0.06cm);
\filldraw (7,0) circle (0.06cm);
\filldraw (8,0) circle (0.06cm);
\filldraw (9,0) circle (0.06cm);
\filldraw (9,0) circle (0.06cm);
\draw [very thick] (0,0) -- (4,0);
\draw [very thick, dashed] (4,0) -- (6,0);
\draw [very thick] (6,0) -- (9,0);
\node at (0, -0.5) {1};
\node at (1, -0.5) {2};
\node at (9,-0.5) {$\mbox{diam}(G)$};
\node at (5,-0.5) {$k$};
\draw [thick] (5,0) to[out=60, in=270] (5.5, 1) to[out=90, in=0] (5, 2) to[out=180, in=90] (4.5, 1) to[out=270, in=120] (5, 0);
\node at (5,1) {$G_{k,1}$};
\draw [thick] (5,0) to[out=30, in=270] (6.5,1) to[out=90, in=0] (6,2) to[out=180, in=90] (5.8, 1.5) to[out=270, in=40] (5,0);
\node at (6.1,1) {$G_{k,2}$};
\end{tikzpicture}
\end{center}
\caption{Augmented path graphs: a long path whose attached graphs are connected to exactly one vertex on the path and do not have any connections between them.}
\label{fig:bubble}
\end{figure}
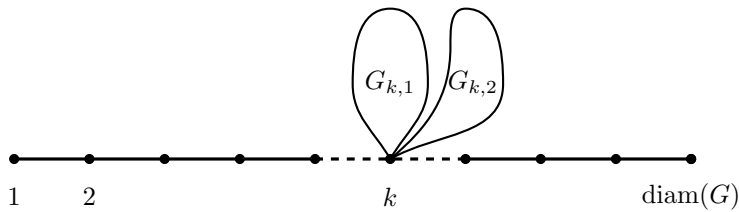

We will now assign to each such subgraph $G_{k,i}$ a natural quantity: for any vertex $v \in G_{k,i}$, we can consider a random walk started in $v$ that jumps uniformly at random to an adjacent vertex until it hits the path at $G_{k,i}(0)= G_p(k)$. We can then, for each such vertex $v$, compute the expected {\em hitting time}.
\begin{definition}
The hitting time $\mbox{hit}(\tilde{G},v_s,v_t)$ is the expected number of random walk steps on the graph $\tilde{G}$ to get from $v_s$ to $v_t$ for the first time.
\end{definition}
With a slight abuse of terminology, we define a hitting time $\mbox{hit}(\tilde{G},v_t)$ for a terminal point $v_t$ as the maximum over starting points in a graph: $$ \mbox{hit}(\tilde{G},v_t) = \max_{v_s \in \tilde{G}} \mbox{hit} (\tilde{G},v_s,v_t) . $$
Finally, we define a hitting time for our construction. 
In the context of this section, we are always interested in the terminal vertex $v_t=G_{k,i}(0)= G_p(k)$ being on the path, and it is therefore convenient to absorb the terminal node in a succinct notation
\begin{equation}\label{eq:path:time}
\mbox{hit}({G_{k,i}}) = \mbox{hit}(G_{k,i},G_{k,i}(0))
\end{equation}
In other words, the hitting time $\mbox{hit}({G_{k,i}})$ of the subgraph ${G_{k,i}}$ in our construction in this section is the maximum expected time to get to the path from a vertex in the subgraph ${G_{k,i}}$. 
This quantity captures an important aspect of the underlying dynamics. We refer to \cite{xiu} where the same quantity has been used in a similar way. We observe that if $G_{k,i}$ is itself a path, then 
\begin{equation*}
  \mbox{hit}(G_{k,i})  \sim \mbox{diam}(G_{k,i})^2.
\end{equation*}

 This scaling follows from observing that the problem is structurally similar to a random walk on the lattice $\mathbb{Z}$ and that the standard random walk on $\mathbb{Z}$
 after $\ell$ random steps has variance $\ell$ (and thus standard deviation $\sim \sqrt{\ell}$; see also \cite{blum}) . \\
 
 We can now state the main result of the section. We refer to Figure \ref{fig:bubble} for a sketch of what these graphs look like.

\begin{theorem} 
Let $G$ be an augmented path graph (see Definition \ref{def:path-based-graph}) with subgraphs $G_{k,i}$ attached to a path $G_p$ of length $\emph{diam}(G)$. 
Suppose that each $G_{k,i}$ satisfies
\begin{enumerate}
\item the attached subgraph $G_{k,i}$ does not have too many vertices:
$$  |G_{k,i}|   \leq \frac{\emph{diam}(G)}{32}.$$
\item and the hitting time (equation (\ref{eq:path:time})) is not too large:
$$  \emph{hit}(G_{k,i}) \leq   \frac{1}{50}  \min\left\{ k, \emph{diam}(G) - k \right\}^{2}.$$
\end{enumerate}
Then any eigenvector corresponding to the second eigenvalue of the graph Laplacian of $G$ assumes its extrema at the endpoints of the long path.
\end{theorem}
The proof is provided in \S 3.3. 
We point out that it would be of interest to obtain inverse results: explicit conditions under which one of the extrema is not attained at the endpoints of the path.\\

We give a sample application of Theorem 2 to Evans' counterexample and consider what he called the Fiedler rose (see Figure \ref{fig:fiedlerrose}): let $G$ denote the Fiedler rose with $n+2$ vertices. If we start in an outermost vertex (one of the vertices at the bottom of Figure \ref{fig:fiedlerrose}), then one step of the random walk leads to the center and the next step leads to the path with likelihood $p=1/(n+1)$. This means that the expected number of steps required until one hits the path is
\begin{align*}
\mbox{hit}(\mbox{Rose}) &= 2 \frac{1}{n+1} + 4 \frac{1}{n+1} \left(1- \frac{1}{n+1}\right) + 6 \frac{1}{n+1} \left(1- \frac{1}{n+1}\right) ^2 + \dots\\
&= \frac{2}{n+1} \sum_{k=0}^{\infty} (k+1) \left(1- \frac{1}{n+1}\right) ^k = 2n +2.
\end{align*}

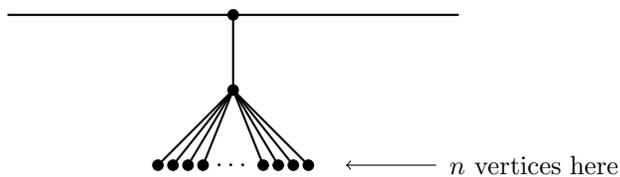
\begin{figure}[h!] 
\begin{center}
\begin{tikzpicture}
\filldraw (3,0) circle (0.07cm);
\draw [thick] (0,0) -- (6,0);
\filldraw (3,-1) circle (0.07cm);
\filldraw (2,-2) circle (0.07cm);
\filldraw (2.2,-2) circle (0.07cm);
\filldraw (2.4,-2) circle (0.07cm);
\filldraw (2.6,-2) circle (0.07cm);
\node at (3,-2) {$\dots$};
\filldraw (3.4,-2) circle (0.07cm);
\filldraw (3.6,-2) circle (0.07cm);
\filldraw (3.8,-2) circle (0.07cm);
\filldraw (4,-2) circle (0.07cm);
\draw [thick] (3,0) -- (3,-1);
\draw [thick] (2,-2) -- (3,-1);
\draw [thick] (2.2,-2) -- (3,-1);
\draw [thick] (2.4,-2) -- (3,-1);
\draw [thick] (2.6,-2) -- (3,-1);
\draw [thick] (3.4,-2) -- (3,-1);
\draw [thick] (3.6,-2) -- (3,-1);
\draw [thick] (3.8,-2) -- (3,-1);
\draw [thick] (4,-2) -- (3,-1);
\node at (7, -2) {$n$ vertices here};
\draw [->] (5.7,-2) -- (4.5, -2);
\end{tikzpicture}
\end{center}
\caption{The `Fiedler rose' counterexample of Evans \cite{evans}.\label{fig:fiedlerrose} }
\end{figure}

This means that if we have a path graph of length $\mbox{diam}(G)$ and attach a Fiedler rose with $n$ vertices to
the middle point of the path graph, then the rose can have up to $n \leq \mbox{diam}(G)/120$ vertices without violating
the result. More precise asymptotics for this special case were given by Lef\`evre \cite{lef}.

\subsection{A Hitting Time Bound.} The purpose of this section is to establish bounds on hitting times under an assumption on the maximum degree. Let $G$ be a connected graph and assume that vertex $v_1$ is marked (in the setting above, $v_1$ is the vertex that lies on the long path). We are interested in obtaining upper bounds on the expected hitting time $\mbox{hit}(G, v_1)$. 
Evans constructed his counterexample using the maximum degree of $G$; the following proposition provides a bound on the hitting time as a function of the maximum degree.

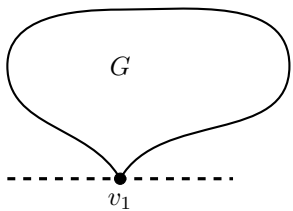
\begin{figure}[h!]
\begin{center}
\begin{tikzpicture}[scale=1.5]
\draw [very thick, dashed] (4,0) -- (6,0);
\node at (5,-0.2) {$v_1$};
\draw [thick] (5,0) to[out=60, in=270] (6.5, 1) to[out=90, in=0] (5, 1.5) to[out=180, in=90] (4, 1) to[out=270, in=120] (5, 0);
\node at (5,1) {$G$};
\filldraw (5,0) circle (0.05cm);
\end{tikzpicture}
\end{center}
\caption{A connected graph $G$ with a marked vertex.}
\label{fig:singlebubble}
\end{figure}
\begin{proposition} Let $G$ be a connected graph with maximal degree $\Delta  = \max_{v \in G}~ \emph{deg}(v)$ and a marked vertex $v_1$. The maximum expected time of a random walk started in a vertex in $G$ until it hits $v_1$ can be bounded from above by
  $$  \emph{hit}(G,v_1)  \leq \emph{diam}(G) \cdot  \Delta^{{\tiny \emph{diam}(G)}}.$$
\end{proposition}

The proof of is presented in \S \ref{pf:prop}.
We can apply the proposition in combination with Theorem 2 to construct a fairly general family of graphs for which the maximum is indeed assumed at the endpoints of the longest path. These graphs are a subset of the augmented path graphs (Definition \ref{def:path-based-graph}) illustrated in Figure \ref{fig:bubble}: there is a long underlying path $G_p$ of length $\mbox{diam}(G)$ whose vertices are enumerated by $\left\{1,2,\dots,\mbox{diam}(G) \right\}$.  If $k$ is a vertex on the path, then we are allowed to attach an arbitrary number of graphs
$G_{k,1}, G_{k,2}, \dots$ to the vertex as long as the only edges between $G_{k,i}$ and the path are adjacent to the vertex $k$. Moreover, no edges between $G_{k,i}$ and $G_{k,j}$ or more, generally, $G_{i,k}$ and $G_{k,\ell}$, are allowed (Figure \ref{fig:entirefamily}).

\begin{figure}[h!]
\begin{center}
\begin{tikzpicture}[scale=1]
\filldraw (0,0) circle (0.06cm);
\filldraw (1,0) circle (0.06cm);
\filldraw (2,0) circle (0.06cm);
\filldraw (3,0) circle (0.06cm);
\filldraw (4,0) circle (0.06cm);
\filldraw (5,0) circle (0.06cm);
\filldraw (6,0) circle (0.06cm);
\filldraw (7,0) circle (0.06cm);
\filldraw (8,0) circle (0.06cm);
\filldraw (9,0) circle (0.06cm);
\filldraw (9,0) circle (0.06cm);
\draw [very thick] (0,0) -- (4,0);
\draw [very thick, dashed] (4,0) -- (8,0);
\draw [very thick] (8,0) -- (9,0);
\node at (0, -0.5) {1};
\node at (1, -0.5) {2};
\node at (9,-0.5) {$\mbox{diam}(G)$};
\node at (5,-0.5) {$k$};
\draw [thick] (5,0) to[out=60, in=270] (5.5, 1) to[out=90, in=0] (5, 2) to[out=180, in=90] (4.5, 1) to[out=270, in=120] (5, 0);
\node at (5,1) {$G_{k,1}$};
\draw [thick] (5,0) to[out=30, in=270] (6.5,1) to[out=90, in=0] (6,2) to[out=180, in=90] (5.8, 1.5) to[out=270, in=40] (5,0);
\node at (6.1,1) {$G_{k,2}$};
\draw [thick] (7,0) to[out=30, in=270] (7.5,0.5) to[out=90, in=0] (7.3,1) to[out=180, in=90] (6.8, 0.8) to[out=270, in=40] (7,0);
\node at (7.2,0.6) {$_{G_{k',2}}$};
\node at (7,-0.5) {$k'$};
\end{tikzpicture}
\end{center}
\caption{An explicit admissible family of graphs: as above, we are allowed to attach graphs $G_{k,i}$ to the vertex $i$ as long as the diameter
of $G_{k,i}$ is short compared to the maximal degree of $G_{k,i}$.}
\label{fig:entirefamily}
\end{figure}
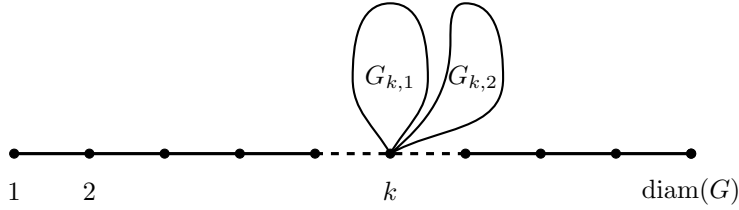

Theorem 2 gives us concrete conditions under which the maximum and minimum are assumed at endpoints of the longest path, these are
\begin{enumerate}
\item $  |G_{k,i}|   \leq (1/32) \cdot \mbox{diam}(G)$
\item 
$  \mbox{hit}(G_{k,i}) \leq   (1/50) \cdot  \min\left\{ k, \mbox{diam}(G) - k \right\}^{2}.$
\end{enumerate}
The second condition may be difficult to check and this is where the Proposition comes into play. If we denote the maximal degree of $G_{k,i}$ by $\Delta_{k,i}$ and the the maximal degree of $G$ by $\Delta$, then
we can use the Proposition to bound
$$ \mbox{hit}(G_{k,i}) \leq \mbox{diam}(G_{k,i}) \cdot  \Delta_{k,i}^{{\tiny \mbox{diam}(G_{k,i})}}  \leq \mbox{diam}(G_{k,i}) \cdot  \Delta^{{\tiny \mbox{diam}(G_{k,i})}}  .$$

To ensure that Condition 2 holds, it thus suffices to assume that
$$  \mbox{diam}(G_{k,i}) \cdot  \Delta^{{\tiny \mbox{diam}(G_{k,i})}} \leq \frac{1}{50} \cdot  \min\left\{ k, \mbox{diam}(G) - k \right\}^{2}.$$
In the middle of the path, $k \sim \mbox{diam}(G)/2$, this shows that 
the attached $G_{k,i}$ in the middle of the path can be any arbitrary subgraph
as long its maximum degree is bounded by $\Delta$ and 
$$  \mbox{diam}(G_{k,i}) \leq  c_{\Delta} \log \mbox{diam}(G),$$
where the constant $c_{\Delta}$ depends only on $\Delta$ -- this suffices to ensure Condition 2 to be satisfied.
This may, at first glance, seem like a rather restricted result: the subgraph  $G_{k,i}$ may be arbitrary as long as its diameter is short. In light of Evans' counterexample, this estimate is perhaps not surprising (one can attach Fiedler roses on top of Fiedler roses on top of Fiedler roses etc. to the desired effect). However, we also point out that if the subgraph  $G_{k,i}$ does not have a `labyrinth-' type structure where random walkers can easily get lost (in the sense of hitting time being large), then one could attach subgraphs of larger diameter without violating the conditions of Theorem 2. This will be investigated in \S 2.5.

\subsection{Caterpillar graphs.} We conclude with a simple example: a caterpillar graph \cite{cat2, cat1} is path of length $n$ where to each vertex we may add trees of size 1 (alternatively: after removing all vertices of degree 1, a path graph remains).  Gernandt \& Pade \cite{gern} proved that the extrema of the second eigenvector are assumed at the endpoints of the longest path and established various generalizations of this result. We give another such result which will follow quickly from Theorem 2.

\begin{corollary} Let $G_p=P_n$ be a path with vertices $1, 2, \dots, n$. Suppose we attach to the vertex $k$
an arbitrary number of paths of length at most $f(k)$, where
$$ f(k) \leq  \frac{1}{20}  \min\left\{ k, n-k\right\}^{}.$$
Then the global extrema are assumed at the endpoints of the longest path.
\end{corollary}
Stronger results should be true, maybe even
$f(k) = \min\left\{k, n-k\right\} - 1$. 

\subsection{A Hitting Time Problem.} 
An interesting question is the following: suppose $G=(V,E)$ is a connected graph with a marked vertex $v_1$ and $h:V \rightarrow \mathbb{R}$
is a function such that $h(v)$ the expected number of steps a random walk started in $v$ takes until it hits $v_1$. What bounds (both from above and from below) can be proven
on 
$$ \mbox{hit}_{v_1}(G) = \max_{v \in V}{h(v)} ?$$
A trivial bound is 
$$ \mbox{hit}_{v_1}(G) \geq \max_{v \in V}{ d(v,v_1)}.$$
Amusingly, this might be close to optimal. Fix a degree $\Delta$ and consider the following type of graph where each vertex has the maximal number of children ($\Delta -1$) up to a certain level. Let us then connect all the vertices in the last level to the root of the tree. The induced random walk can be regarded as a biased random walk in terms of the level and will quickly lead to the root of the tree.

\begin{figure}[h!] 
\begin{center}
\begin{tikzpicture}
\filldraw (0,0) circle (0.07cm);
\filldraw (1,1) circle (0.07cm);
\filldraw (1,0) circle (0.07cm);
\filldraw (1,-1) circle (0.07cm);
\draw [thick] (0,0) -- (1,1);
\draw [thick] (0,0) -- (1,0);
\draw [thick] (0,0) -- (1,-1);
\filldraw (2,1.3) circle (0.07cm);
\filldraw (2,1) circle (0.07cm);
\filldraw (2,0.7) circle (0.07cm);
\draw [thick] (1,1) -- (2,1.3);
\draw [thick] (1,1) -- (2,1);
\draw [thick] (1,1) -- (2,0.7);
\filldraw (2,0.3) circle (0.07cm);
\filldraw (2,0) circle (0.07cm);
\filldraw (2,-0.3) circle (0.07cm);
\filldraw (2,-1.3) circle (0.07cm);
\filldraw (2,-1) circle (0.07cm);
\filldraw (2,-0.7) circle (0.07cm);
\draw [thick] (1,0) -- (2,0.3);
\draw [thick] (1,0) -- (2,0);
\draw [thick] (1,0) -- (2,-0.3);
\draw [thick] (1,-1) -- (2,-1.3);
\draw [thick] (1,-1) -- (2,-1);
\draw [thick] (1,-1) -- (2,-0.7);
\foreach \x in {0,...,12}
{
\filldraw (4,1.5-\x/4) circle (0.07cm);
};
\node at (3,0) {$\dots$};
\end{tikzpicture}
\end{center}
\caption{An example with an induced drift on the random walk.}
\end{figure}
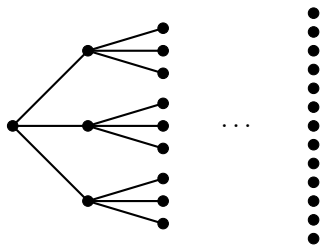
A simple question is the following: what sort of hitting time bounds are possible and how do they depend on the graph. For example, if $G$ is a tree, then we have
$ \mbox{hit}_{v_1}(G) \gtrsim \mbox{diam}(G)^2$. What other results are possible?

\section{Proofs}
\subsection{Proof of Theorem 1}\label{sec:pf:1}

\begin{proof}
We first note that basic Markov Chain theory implies that for each vertex in the graph, the game has an expected value: the graph is finite and connected and therefore a random walk will eventually hit every vertex almost surely. More precise results would be possible: in particular, the likelihood of a random walk taking a very long time before hitting $v_t$ decays exponentially (with a constant depending on the graph), see for example Levin, Peres \& Wilmer \cite{levin}, but this will not be needed here.
Let us denote by $\psi(v)$ the expected value of the game terminating at $v_t$ when starting in a vertex $v \in V$, i.e., $\psi(v) = \mathbb{E}\left(\mbox{payoff}(v \rightarrow v_t)\right)$. 
 Then, by definition of the game, 
\begin{equation}\label{eq:pf1:phi_vt}
 \psi(v_t) =0
\end{equation}
since the game terminates immediately when started at that vertex. 
Let now $v \neq v_t$. By the structure of the game, we are able to relate $\psi(v)$ to the value of $\psi$ in all
adjacent neighbors via the equation
\begin{equation}\label{eq:game:eq}
 \psi(v) = \frac{\lambda_{k}\cdot\phi_k(v)}{\mbox{deg}(v)}  + \frac{1}{\mbox{deg}(v)} \sum_{v \sim_E w} \psi(w). 
\end{equation}

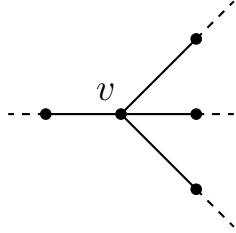
\begin{figure}[h!] 
\begin{center}
\begin{tikzpicture}
\filldraw (1,1) circle (0.07cm);
\filldraw (0,0) circle (0.07cm);
\filldraw (1,0) circle (0.07cm);
\filldraw (1,-1) circle (0.07cm);
\filldraw (-1,0) circle (0.07cm);
\draw [thick] (0,0) -- (1,1);
\draw [thick] (0,0) -- (1,0);
\draw [thick] (0,0) -- (-1,0);
\draw [thick] (0,0) -- (1,-1);
\draw [thick, dashed] (-1,0) -- (-1.5, 0);
\draw [thick, dashed] (1,0) -- (1.5, 0);
\draw [thick, dashed] (1,1) -- (1.5, 1.5);
\draw [thick, dashed] (1,-1) -- (1.5, -1.5);
\node at (-0.2, 0.3) {\LARGE $v$};
\end{tikzpicture}
\end{center}
\caption{The idea behind the proof: relating everything to neighbors}
\end{figure}

We introduce the function $h:V \rightarrow \mathbb{R}$ given by
$$ h(v) = \phi_k(v) - \psi(v).$$
We conclude from equation (\ref{eq:pf1:phi_vt}) that $h(v_t) = \phi_k(v_t)$.
We recall the definition of the Laplacian (\ref{eq:lap}) and its eigenvector $\phi_k,$ with eigenvalue $\lambda_k$,
$$ (D-A) \phi_k = \lambda_k \phi_k,$$
we see that $\phi_k$ satisfies, in all vertices $v$, 
\begin{equation}\label{eq:lap:eig:explicit}
 \phi_k(v) =  \frac{\lambda_{k}\cdot\phi_k(v)}{\mbox{deg}(v)}  + \frac{1}{\mbox{deg}(v)} \sum_{v \sim_E w} \phi_k(w).
\end{equation}
Subtracting (\ref{eq:game:eq}) from (\ref{eq:lap:eig:explicit}), we see that $h$ satisfies, for all vertices $v \in V \setminus \left\{ v_t \right\}$ that
\begin{equation}\label{eq:h:neigh}
 h(v) = \frac{1}{\mbox{deg}(v)} \sum_{v \sim_E w}  h(w).
\end{equation}
We will now prove that $h$ is constant, $h \equiv \phi_k(v_t)$ which establishes the result. Suppose now there exists $v \in V$ such that $h(v) > \phi_k(v_t)$ (the case $h(v) < \phi_k(v_t)$ is analogous). Let us define 
$$ m = \max_{v \in V} h(v) > \phi_k(v_t).$$
Let $v$ be a vertex such that $h(v) = m$. Then 
$$ m = h(v) = \frac{1}{\mbox{deg}(v)} \sum_{v \sim_E w}  h(w) \leq \frac{1}{\mbox{deg}(v)} \sum_{v \sim_E w}  m = m$$
and necessarily $h(w) = m$ for all neighbors of $v$. The graph is connected, therefore $h \equiv m$ in $V \setminus \left\{v_t \right\}$ and $0$ in $v_t$.
However, then it is easy to see that $h$ does not satisfy the equation
$$ h(v) = \frac{1}{\mbox{deg}(v)} \sum_{v \sim_E w}  h(w)$$
in all vertices adjacent to $v_t$. This contradiction establishes that $h \equiv \phi_k(v_t)$ and thus that 
$$\psi(v) \equiv \phi_k(v) - \phi_k(v_t).$$ 
This shows that if we start the game in an arbitrary vertex $v_s \neq v_t$, the expected payoff is $\psi(v_s) = \phi_k(v_s) - \phi_k(v_t)$.
\end{proof}

\subsection{Some Preliminary Considerations}
Before embarking on the proof of Theorem 2, we recall several helpful statements and derive some basic facts for the graphs under consideration. Many
of these facts are either well known (see e.g. Mohar \cite{mohar2, mohar}) or folklore; we recall them for clarity of exposition.

\begin{lemma} Let $G$ be a graph satisfying all the conditions in Theorem 2. Then
 \begin{equation}\label{eq:eig2bound}
 \lambda_2(G) \leq \frac{10}{\emph{diam}(G)^2}.
 \end{equation}
\end{lemma}

\begin{proof}
We use the standard variational characterization of the second eigenvector of the Laplacian 
\begin{equation}\label{eq:lem1:variational}
 \lambda_{2}(G) = \min_{ \left\langle x , \mathbf{1} \right\rangle = 0} \frac{\sum_{v_i \sim_E v_j}{(x_i - x_j)^2}}{\sum_{i=1}^{n}{x_i^2}},
 \end{equation}
 where the minimum is taken with respect to all vectors $x \not\equiv \mathbf{0}$ having mean value 0. 
A standard result from spectral graph theory \cite{chung, fallat} states that the second eigenvalue $\lambda_2 (P_n)$ of a path of length $n$ is:
 \begin{align*}
\lambda_2(P_n)= 2 \left( 1 - \cos{\left( \frac{\pi}{n} \right)} \right).
  \end{align*}
The path $G_p$ of the augmented path graph $G$ is itself a path of length $n=\mbox{diam}(G)$. Using $\lambda_2(P_{\tiny \mbox{diam}(G)})$ and $\varphi$ to denote the second eigenvalue of $G_p$ and its $\ell^2-$normalized second eigenvector:
 \begin{equation}
\begin{split}
  \sum_{i,i+1 ~{\tiny \mbox{on long path $G_p$}}}{(\varphi(i) - \varphi(i+1))^2} &=  \lambda_2(P_{\tiny \mbox{diam}(G)})\\
  &= 2 \left( 1 - \cos{\left( \frac{\pi}{\mbox{diam}(G)} \right)} \right) \leq \frac{10}{\mbox{diam}(G)^2}.
\end{split}
  \end{equation}

We define $p:V \rightarrow \mathbb{R}$ to be the $\ell^2-$normalized second
 eigenvector of the path graph defined on the long path $G_p$ (in particular, $p$ is 0 on vertices that do not lie on the path): 
  \begin{equation}\label{eq:lem1:q}
	p(v) = \begin{cases} \varphi(v) & \text{if~} v \text{~on~path~} G_p \\
	0 & \text{o.w.}
	\end{cases}
\end{equation}
Next, we define the
 vector $q:V \rightarrow \mathbb{R}$ to be 0 on the path and have the constant value corresponding to the value of $p$ in each bubble, i.e.
 \begin{equation}\label{eq:lem1:q}
	q(v) = \begin{cases} p(k) & \text{if~} v \in G_{k,i}\backslash V_{k,i}(0) \\
	0 & \text{o.w.}
	\end{cases}
\end{equation}
We note that the supports of $q$ and $p$ do not overlap, and therefore $q$ is orthogonal to $p$.
We start by considering the vector
$$ z = p+q.$$
 We first note that  when computing the quadratic form $\left\langle z, L z\right\rangle$ (see Equation (\ref{eq:quad})), the only contribution comes from adjacent edges on the path
 graph: the function is constant between all other edges.   Therefore, 
 $$ \sum_{v_i \sim_E v_j}{(z_i - z_j)^2} = \sum_{i,i+1 ~{\tiny \mbox{on long path}}}{(\varphi(i) - \varphi(i+1))^2} = \lambda_2(P_{\tiny \mbox{diam}(G)}) 
 \leq \frac{10}{\mbox{diam}(G)^2}.$$

However, 
$z$ does not have mean value 0. We introduce the orthogonal projection onto vectors with mean value 0 
$$ y =z - \left\langle z, \frac{1}{\sqrt{|V|}}\right\rangle \frac{1}{\sqrt{|V|}}$$
and note that, since we subtract the constant vector,
$$ \left\langle y, Ly\right\rangle = \sum_{v_i \sim_E v_j}{(y_i - y_j)^2} = \sum_{v_i \sim_E v_j}{(z_i - z_j)^2} = \left\langle z, Lz\right\rangle \leq  \frac{10}{\mbox{diam}(G)^2}.$$
It remains to show that $\|y\|^2 \geq 1$ to conclude the result. Using, in that order, the Pythagorean theorem (twice: first for projections, then to evaluate the norm of $z=p+q$), the fact that $p$ has mean value 0, the
Cauchy-Schwarz inequality and the $\ell^2-$normalization of $p$, we arrive at
\begin{align*}
\|y\|^2 &= \|z\|^2 - \left\langle z, \frac{1}{\sqrt{|V|}}\right\rangle^2 \\
&= \|p\|^2 + \|q\|^2 -  \left\langle z, \frac{1}{\sqrt{|V|}}\right\rangle^2 \\
&=  \|p\|^2 + \|q\|^2 -  \left\langle q, \frac{1}{\sqrt{|V|}}\right\rangle^2\\
& \geq \|p\|^2 + \|q\|^2 -  \|q\|^2 = \|p\|^2 = 1.
\end{align*}

Now that we have $y$ such that  $\left\langle y , \mathbf{1} \right\rangle = 0$ and $\|y\|^2 \geq 1$,
\begin{equation}
 \lambda_{2}(G) = \min_{ \left\langle x , \mathbf{1} \right\rangle = 0} \frac{\sum_{v_i \sim_E v_j}{(x_i - x_j)^2}}{\sum_{i=1}^{n}{x_i^2}} \leq \frac{\left\langle y, L y\right\rangle}{\left\langle y, y\right\rangle} \leq  \frac{10}{\mbox{diam}(G)^2}.
 \end{equation}
\end{proof}

  We note that the proof only uses the ``bubble'' structure  and the fixed diameter of augmented path graphs specified in Theorem 2, and therefore the Lemma applies to more general cases of augmented path graphs. 
 Furthermore, we note that similar bounds on the eigenvalues of graphs to which isolated ``bubbles'' are added can be constructed using the original eigenvalues of the graphs without the bubbles, and the construction proposed in this proof.
 This result is optimal up to constants which can be seen as follows: a result of McKay \cite{mohar2} states that
$$ \lambda_2 \geq \frac{4}{|V| \cdot \mbox{diam}(G)}.$$
McKay's bound shows that this upper bound using only the diameter is optimal up to constants for graphs for which $|V| \sim \mbox{diam}(G)$. The next ingredient that we establish is an upper bound on
the maximum size of $\phi_2.$

\begin{lemma} Let $G$ be a graph satisfying all the conditions in Theorem 2. Denote by $\phi_2$ the $\ell^2-$normalized second eigenvector of the graph Laplacian of $G$.  Then 
$$ \max_{v \in V} |\phi_2(v)| \leq  \frac{4}{\emph{diam}(G)^{1/2}}.$$
\end{lemma}
\begin{proof}
We observe that $\phi_2$ has both positive and negative values because it is orthogonal to $\phi_1$ which is a constant vector.
Since we can replace $\phi_2$ without loss of generality by $-\phi_2$, it suffices to bound the maximum from above. Let $\phi_2(v) > 0$ be arbitrary. We use to $\pi$ to denote a path from $\phi_2$ to the nearest vertex $w$ where $\phi_2(w) \leq 0$.
 For a normalized eigenvector $\phi_2$, this shows that
\begin{align*}
 \max_{v \in V}{ | \phi_2(v) |} &\leq \sum_{(i,j) \in \pi}{ | \phi_2(i) - \phi_2(j)|} \\
 &\leq \left( \sum_{(i,j) \in \pi}{ | \phi_2(i) - \phi_2(j)|^2}\right)^{1/2} \left( \mbox{length of}~\pi\right)^{1/2}\\
 &\leq \left( \sum_{(i,j) \in E}{ | \phi_2(i) - \phi_2(j)|^2}\right)^{1/2} \mbox{diam}(G)^{1/2} \\
 &\leq \lambda_2^{1/2} \mbox{diam}(G)^{1/2} \leq \left( \frac{10}{\mbox{diam}(G)^2}\right)^{1/2} \mbox{diam}(G)^{1/2}\\
 &\leq \frac{4}{\mbox{diam}(G)^{1/2}},
\end{align*}
where the second line uses the Cauchy--Schwarz inequality, the fourth line uses equation (\ref{one}) and the fifth line uses equation (\ref{eq:eig2bound}). 
Since this holds for every vertex, we have 
\begin{equation}\label{eq:linf}
  \| \phi_2\|_{\ell^{\infty}} \leq \frac{4}{\mbox{diam}(G)^{1/2}}.
  \end{equation}
\end{proof}

We use this inequality to derive a lower bound on $\| \phi_2\|_{\ell^1(V)}$.
\begin{lemma} Let $G$ be a graph satisfying all the conditions in Theorem 2.  Denote by $\phi_2$ the $\ell^2-$normalized second eigenvector of the graph Laplacian of $G$. 
Then 
$$ \sum_{v \in V} |\phi_2(v)| \geq  \frac{\emph{diam}(G)^{1/2}}{4} \| \phi_2\|_{\ell^2}.$$
\end{lemma}
\begin{proof} Let us assume w.l.o.g. that $\| \phi_2\|_{\ell^2} = 1$. 
The normalization in $\ell^2$ of  $\phi_2$ implies by the H\"{o}lder inequality that
$$ 1 = \sum_{v \in V}{\phi_2(v)^2} \leq \| \phi_2\|_{\ell^{\infty}} \| \phi_2\|_{\ell^1} $$
and therefore, by equation (\ref{eq:linf}):
$$ \|\phi_2\|_{\ell^1} \geq \frac{\mbox{diam}(G)^{1/2}}{4}.$$
\end{proof}
We note that since $\phi_2$ has mean value 0 (it is orthogonal to the constant $\phi_1$), the positive part and the negative part cancel out and therefore
\begin{equation}\label{eq:maxphi}
  \sum_{v \in V}{ \max\left\{ \phi_2(v), 0 \right\}} \geq \frac{\mbox{diam}(G)^{1/2}}{8}.
\end{equation}
As the last ingredient, we refer to the known result (see e.g. \cite{blum}) about the hitting time of a path $P_k$
\begin{equation}\label{eq:blum}
  \mbox{hit}(P_k) = (k-1)^2.
\end{equation}

\subsection{Proof of Theorem 2}
\begin{proof}
The proof has two parts: first, we show that there are vertices $v_1, v_2$ on the long path $G_p$ of $G$ such that
\begin{equation}\label{eq:pf2:cases} 
\phi_2(v_1) < 0 < \phi_2(v_2).
\end{equation}
 In the second part of the proof, we show that the maximum and minimum are attained at the endpoints. Both parts of the proof makes use of the game interpretation in Theorem 1. We assume throughout the proof that $\phi_2$ is a fixed $\ell^2-$normalized eigenvector associated with the eigenvalue $\lambda_2 > 0$. We make use of the notation introduced in Definition \ref{def:path-based-graph} to reference particular subgraphs and vertices.

{\bf Part 1:}
The argument exploits the bounds derived in \S 3.2. Suppose the statement (\ref{eq:pf2:cases}) is false. Then, for all vertices $v$ on the long path $G_p$, we either have $\phi_2(v) \geq 0$ or $\phi_2(v) \leq 0$. Without loss of generality (after possibly replacing $\phi_2$ by $-\phi_2$) we can assume that for all vertices $v$ on the long path $G_p$, we have $\phi_2(v) \leq 0$. However, since $\phi_2$ has mean value 0, we must conclude that there are positive values somewhere else. 
Naturally, this has to happen inside (some of) the bubbles $G_{k,i}$. 
We recall that by Fiedler's theorem, we have that
$$ \left\{v \in V: \phi_2(v) \geq 0 \right\} \qquad \mbox{is connected.}$$
We first deal with the case where the positive values are all assumed within exactly one bubble $G_{k,i}$.
Then, however, by Equation (\ref{eq:maxphi}) and using Equation (\ref{eq:linf}), we have 
\begin{align*}
 \frac{\mbox{diam}(G)^{1/2}}{8} &\leq \sum_{v \in V}{ \max\left\{ \phi_2(v), 0 \right\}}\\
 &= \sum_{v \in G_{k,i}}{ \max\left\{ \phi_2(v), 0 \right\}} \\
 &\leq |G_{k,i}| \cdot \|\phi_2\|_{\ell^{\infty}} \leq \frac{4 \cdot |G_{k,i}|}{\mbox{diam}(G)^{1/2}}
\end{align*}
and thus
$$ |G_{k,i}| \geq \frac{\mbox{diam}(G)}{32}$$
which is a contradiction to assumption (1) of Theorem 2. This shows that it is not possible for all the positive values to be assumed within one bubble $G_{k,i}$.

\begin{figure}[h!]
\begin{center}
\begin{tikzpicture}
\filldraw (4,0) circle (0.06cm);
\filldraw (5,0) circle (0.06cm);
\filldraw (6,0) circle (0.06cm);
\draw [thick] (2,0) -- (4,0);
\draw [very thick] (4,0) -- (6,0);
\draw [thick] (6,0) -- (8,0);
\draw [thick] (5,0) to[out=60, in=270] (5.5, 1) to[out=90, in=0] (5, 2) to[out=180, in=90] (4.5, 1) to[out=270, in=120] (5, 0);
\node at (5,1) {$G_{k,i}$};
\draw [thick] (5,0) to[out=30, in=270] (6.5,1) to[out=90, in=0] (6,2) to[out=180, in=90] (5.8, 1.5) to[out=270, in=40] (5,0);
\draw [thick] (4,0) to[out=60, in=270] (4.3, 1) to[out=90, in=0] (4, 2) to[out=180, in=90] (3, 1) to[out=270, in=120] (4, 0);
\node at (6.1,1) {$G_{k,j}$};
\node at (3.7, 1) {$G_{l,m}$};
\end{tikzpicture}
\end{center}
\caption{If $\phi_2 \leq 0$ on the long path, then the set of vertices $\phi_2(v) = 0$, if non-empty, has to be group of consecutive vertices. }\label{fig:thm2part15}
\end{figure}
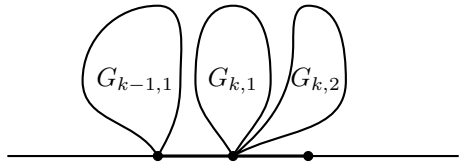

Therefore the set $\left\{v \in V: \phi_2(v) > 0\right\}$ must have elements from at least two different bubbles. We first consider the case that both bubbles are associated with the same vertex {$V_p(k)$} on the long path, {we denote two of these bubbles by} $G_{k,i}$ and $G_{k,j}$. Since {$\phi_2\left(V_p(k)\right) \leq 0$} and since $ \left\{v \in V: \phi_2(v) \geq 0 \right\}$ is connected, we infer that {$\phi_2\left(V_p(k)\right) = 0$}. This observation generalizes: if $\left\{v \in V: \phi_2(v) > 0\right\}$ has elements from different bubbles that are connected to different vertices on the long path, $G_{k,i}$ and $G_{\ell, m}$, then we can again infer that $\phi_2$ vanishes on the segment of the long path connecting $k$ and $\ell$ (see {Figure} \ref{fig:thm2part15}).\\

In either case, we observe that each bubble $G_{k,i}$ containing vertices where $\phi_2(v) > 0$ is necessarily connected to a vertex {$V_p(k)$} on the long path for which {$\phi_2\left(V_p(k)\right)=0$}.
Let now $G_{k,1}$ be an arbitrary bubble containing positive values: we can then apply the game (Theorem 1) starting in the vertex of $G_{k,1}$ with the largest entry and ending in the vertex $V_p(k)$. We recall that the game would involve taking an expected number of $\mbox{hit}(G_{k,1})$ steps (defined in Equation (\ref{eq:path:time})) before it terminates, at each step it accumulates  $\lambda_{2} \phi_2(w)/\mbox{deg}(w) \leq \lambda_{2} \left(\max_{v \in V_{k,1}} \phi_2(v) \right)$ for the vertex $w$ it visits in that step. We conclude that
\begin{align}\label{eq:thm2:gamegraph}
\max_{v \in V_{k,1} } \phi_2(v) - \phi_2(V_p(k)) &= \mathbb{E} \left(\mbox{payoff}\right)\\
 &\leq  \lambda_{2} \cdot  \left(\max_{v \in V_{k,1}} \phi_2(v) \right)  \cdot \mbox{hit}(G_{k,1})
\end{align}
where {$V_p(k)$} is the vertex on the path where $G_{k,1}$ is connected to the path $G_p$.
Recall that, by Equation (\ref{eq:eig2bound}), $\lambda_2(G) \leq 10 \cdot \mbox{diam}(G)^{-2}$ and, by assumption 2 of Theorem 2,
$$ \mbox{hit}(G_{k,1}) \leq \frac{\mbox{diam}(G)^2}{20}.$$
These imply that
\begin{equation} \label{eq:thm2:lambdahit}
  \lambda_2(G)\cdot \mbox{hit}(G_{k,1})  \leq \frac{1}{2}
  \end{equation}
and therefore, by Equation (\ref{eq:thm2:gamegraph}),
\begin{align} 
\max_{v \in V_{k,1}} \phi_2(v)  &\leq \frac{\phi(V_p(k))}{1 - \lambda_2 \cdot \mbox{hit}(G_{k,1})  } 
\end{align}
with a positive denominator by Equation (\ref{eq:thm2:lambdahit}), and a non-positive numerator by our assumption for this case, which implies,
\begin{align} \label{eq:thm2:upperphi}
\max_{v \in V_{k,1}} \phi_2(v)  &\leq  0,
\end{align}

Then, however, it is not possible for $G_{k,i}$ to contain any positive entries. This contradiction concludes the first part of the argument.\\

{\bf Part 2:} It remains to show that the maximum and the minimum occur at the endpoints of the path. By symmetry, it suffices to show this for the maximum, the case of the minimum is completely analogous. After possible changing the orientation of the path $G_p$, 
we can infer from part 1 that there are two adjacent vertices $V_p(m)$ and $V_p(m+1)$ such that
$$ \phi_2(V_p(m)) > 0 \geq \phi_2(V_p(m+1)).$$
By Fielder's theorem, all the vertices $V_p(k)$ of the path $G_p$ with $k \leq m$, and all the vertices of all the bubbles bubbles $G_{k,i}$ with $k \leq m$, connected to these vertices of the path, all have a positive value of $\phi_2$. 
We recall again the game from theorem $1$, and observe that any game originating at $V_p(i)$ and terminating at $V_p(j)$ with $i<j<m$ can only accumulate positive values at each step since $\phi_2$ is positive. We conclude that the values along the long path $G_p$ are decreasing (w.l.o.g. and possibly with a change of direction of the path):
\begin{equation}\label{eq:path:monotone}
 \phi_2(V_p(1)) \geq \phi_2(V_p(2)) \geq \dots \geq \phi_2(V_p(m)).
\end{equation}
This shows that the maximum on the long path is assumed at the end-point. 

It remains to exclude the case where the maximum is assumed inside a bubble. Let us assume that the maximum is assumed in $G_{q,1}\backslash V_p(k)$. Then, again by Fiedler's theorem, we have $q \leq m$.
We now play the game (Theorem 1) twice: first to obtain an upper bound on the maximum of $\phi_2$ (under the assumption 
that this maximum is assumed in $G_{q,1}$) and then to obtain a lower bound on $\phi_2$.
First, we consider a game that starts at the vertex that assumes the maximum value of $\phi_2$. Again, we consider an upper bound on what the game can accumulate at each of the expected number of $\mbox{hit}(G_{q,1})$ steps that it takes  (defined in Equation (\ref{eq:path:time})):
\begin{align}\label{eq:thm2:p2:gamegraph}
\max_{v \in V_{q,1} } \phi_2(v) - \phi_2(V_p(q)) &= \mathbb{E} ~\mbox{payoff}\\
 &\leq  \lambda_{2} \cdot  \left(\max_{v \in G_{q,1}} \phi_2(v) \right)  \cdot \mbox{hit}(G_{q,1})
\end{align}
where $V_p(q)$ is the vertex on the path $G_p$ where $G_{q,1}$ is connected to the path (by assumption, the maximum is not assumed on the path).

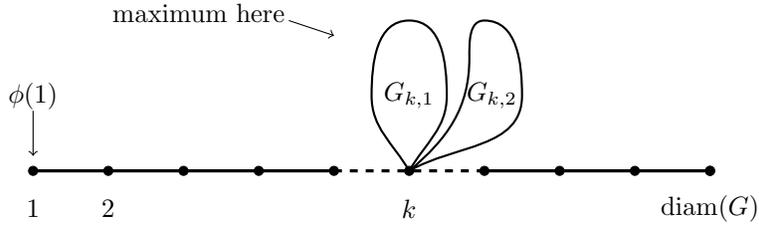
\begin{figure}[h!]
\begin{center}
\begin{tikzpicture}
\filldraw (0,0) circle (0.06cm);
\filldraw (1,0) circle (0.06cm);
\filldraw (2,0) circle (0.06cm);
\filldraw (3,0) circle (0.06cm);
\filldraw (4,0) circle (0.06cm);
\filldraw (5,0) circle (0.06cm);
\filldraw (6,0) circle (0.06cm);
\filldraw (7,0) circle (0.06cm);
\filldraw (8,0) circle (0.06cm);
\filldraw (9,0) circle (0.06cm);
\filldraw (9,0) circle (0.06cm);
\draw [very thick] (0,0) -- (4,0);
\draw [very thick, dashed] (4,0) -- (6,0);
\draw [very thick] (6,0) -- (9,0);
\node at (0, -0.5) {1};
\node at (1, -0.5) {2};
\node at (9,-0.5) {$\mbox{diam}(G)$};
\node at (5,-0.5) {$k$};
\draw [thick] (5,0) to[out=60, in=270] (5.5, 1) to[out=90, in=0] (5, 2) to[out=180, in=90] (4.5, 1) to[out=270, in=120] (5, 0);
\node at (5,1) {$G_{q,1}$};
\draw [thick] (5,0) to[out=30, in=270] (6.5,1) to[out=90, in=0] (6,2) to[out=180, in=90] (5.8, 1.5) to[out=270, in=40] (5,0);
\node at (6.1,1) {$G_{q,2}$};
\node at (0,1) {$\phi(1)$};
\draw [->] (0,0.8) -- (0,0.2);
\node at (2.2,2.1) {maximum here};
\draw [->] (3.4, 2) -- (4,1.8);
\end{tikzpicture}
\end{center}
\caption{Setup of part 2 of the proof }\label{fig:thm2part2}
\end{figure}

As in part 1, Equation (\ref{eq:thm2:p2:gamegraph}) implies that
\begin{align} \label{eq:thm2:upperphi}
\max_{v \in V_{q,1}} \phi_2(j)  &\leq \frac{\phi_2(V_p(q))}{1 - \lambda_2 \cdot \mbox{hit}(G_{q,1})  },
\end{align}
with a positive denominator by Equation (\ref{eq:thm2:lambdahit}).
We now start the game of  Section \ref{sec:game} in the first vertex $V_p(1)$ of the long path $V_p$ and obtain
$$  \phi_2(V_p(1)) - \phi_2(V_p(q)) = \mathbb{E}~ \mbox{payoff} \left(  V_p(1) \rightarrow V_p(q) \right).$$
We recall that in the game we jump around randomly and add
$$ \lambda_2 \frac{\phi_2(v)}{\mbox{deg}(v)} \qquad \mbox{at every step.}$$

{\bf Reduced case: a trivial path from $1$ to $q$.} Let us first consider a simplified case where the vertices $V_p(1)$ to $V_p(q-1)$ do not have any bubbles attached to them.
In that case, the degree of each vertex leading from $V_p(1)$ to $V_p(q)$ is at most $2$. Since the values of $\phi_2$ on this path are monotone decreasing (Equation (\ref{eq:path:monotone})), each step in the game that starts at $V_p(1)$ and ends at $V_p(q)$ contributes at least $\lambda_k \phi_2(V_p(q)) /2$.
Denoting the expected number of steps of a random walk on a path of length $q$ by $\mbox{hit}(P_q)$, and using Theorem 1, we obtain
\begin{equation}
	\phi_2(V_p(1)) - \phi_2(V_p(q)) \geq \frac{1}{2} \lambda_2\cdot \phi_2(V_p(q)) \cdot \mbox{hit}(P_q),
\end{equation}
and therefore,
\begin{equation}\label{eq:thm2:p2:path}
	\phi_2(V_p(1)) \geq \phi_2(V_p(q)) \left( 1+  \frac{1}{2} \lambda_2 \cdot \mbox{hit}(P_q) \right).
\end{equation}

By the assumption for part 2, $\phi_2(V_p(1)) < \max_{v \in V_{q,1} } \phi_2(v)$ and therefore, using equations (\ref{eq:thm2:upperphi}) and (\ref{eq:thm2:p2:path}),
\begin{equation}
\phi_2(V_p(q)) \left( 1+  \frac{1}{2} \lambda_2 \mbox{hit}(P_q) \right) < \frac{\phi_2(V_p(q))}{1 - \lambda_2 \cdot \mbox{hit}(G_{q,1})  }
\end{equation}
with $\phi_2(V_p(q))  > 0$ in this part of the proof. 
Since the denominator is positive we have
\begin{equation}
 1+\frac{1}{2} \lambda_2  \cdot \mbox{hit}(P_q) - \lambda_2 \cdot \mbox{hit}(G_{q,1}) - \frac{1}{2} \lambda_2^2 \cdot \mbox{hit}(G_{q,1})  \cdot \mbox{hit}(P_q) \leq  1 ,
  \end{equation}
and therefore
\begin{equation}
 \mbox{hit}(G_{q,1}) \geq     \frac{\mbox{hit}(P_q)}{2}    - \frac{\lambda_2  \cdot \mbox{hit}(P_q) \mbox{hit}(G_{q,1})} {2} .
  \end{equation}

By Equation (\ref{eq:eig2bound}) and Equation (\ref{eq:blum}),
\begin{equation}
\lambda_2 \cdot \mbox{hit}(P_q) = \lambda_2 \cdot (q-1)^2 \leq  \lambda_2 \cdot \mbox{diam}(G)^2 \leq 10
  \end{equation}
Therefore,
\begin{equation}
 \mbox{hit}(G_{k,1}) \geq     \frac{\mbox{hit}(P_q)}{2}     - 5 \cdot \mbox{hit}(G_{q,1}) .
  \end{equation}
Thus,
\begin{equation}
 \mbox{hit}(G_{q,1}) \geq   \frac{  \mbox{hit}(P_q)}{12}.
  \end{equation}

Using the hitting time bound (\ref{eq:blum}) and assumption (2) of Theorem 2,
\begin{equation}
\frac{q^2}{50} \geq \mbox{hit}(G_{q,1}) \geq   \frac{  \mbox{hit}(P_q)}{12} = \frac{(q-1)^2}{12},
  \end{equation}
which is a contradiction. \\

{\bf General case: bubbles attached to the path from $1$ to $q$.} After treating the the case of no bubbles attached to the path between $V_p(1)$ to $V_p(q-1)$, we return to the general case.
We will compare the game in the general case to the game in the case of no bubbles, and show that that adding a bubble to the path only increases the payoff from the game, and therefore the proof for the reduced case holds for the general case.
In the reduced case, where there are no bubbles, each vertex from $V_p(2)$ to $V_p(q-1)$ on the path has degree $\mbox{deg}(v) = 2$ (and $V_p(1)$ has $\mbox{deg}(v) = 1$). 
We observed that each visit to a vertex $w$ on a path in the reduced case with no bubbles contributes at least $\lambda_k \phi_2(V_p(w)) /2$ to the payoff in the game, and that this must be a positive number.
Now suppose that some of the vertices $V_p(2)$ to $V_p(q-1)$ have $\mbox{deg}(v) \geq 3$. 
Let $v \in \{ V_p(2) ,..., V_p(q-1) \}$ be a vertex on the path, with $\mbox{deg}(v) \geq 3$.
Once the game enters the vertex $v$, we can examine how many steps of the game it would take before it leaves $v$ and the bubbles associated with $v$, and more on to some other vertex on the path. 
We have encountered a similar calculation in our discussion of the Fiedler rose above.
 Let us denote $r = (\mbox{deg}(v)-2)/\mbox{deg}(v)$, then 
\begin{align*}
\mathbb{E}~\mbox{\# encounters before moving on } &\geq 1(1-r) + 2 r (1-r) + 3 r^2 (1-r) + \dots \\
&= \sum_{j=1}^{\infty}{ j (1-r) r^{j-1}} = \frac{1}{1-r} = \frac{\mbox{deg}(v)}{2}.
\end{align*}
It follows that each time the game visits the set of vertex on the path and its bubbles contributes at least $\lambda_k \phi_2(V_p(q)) /2$, with the minimum contribution of $\lambda_k \phi_2(V_p(q)) /2$ is achieved when the vertex on the path has no bubbles attached to it. 
From this point we can proceed with the same argument as the one made for the reduced case above. 
We conclude that the maximum cannot be achieved inside any of the bubbles, and due to the monotonicity along the path, the maximum must be achieved at the end of the path.

\end{proof}

\subsection{Proof of the Proposition.} \label{pf:prop}

\begin{proof}
  We give a simple estimate that does not yield the sharp constant (for which we refer to Aldous \& Fill \cite{aldous}), however, the argument is very short and very simple. We can label each vertex based on the length of the shortest path from it to $v_1$, since the graph is connected, this label would be a well defined finite number which is at most $\mbox{diam}(G)$. Each vertex in the graph (other than $v_1$) is connected to at least one other vertex that is closer to $v_1$ (i.e., with a lower label) because there is always a shortest path from any vertex to $v_1$. Therefore, at each step of the random walk we decrease the distance to $v_1$ with probability of at least $1/\mbox{deg}(v) \geq \Delta^{-1}$ (until $v_1$ is reached).
 This means the likelihood of hitting $v_1$ within $\mbox{diam}(G)$ steps is at least 
  $$p = \Delta^{-\mbox{\tiny diam}(G)}$$
  which is simply the likelihood of picking a `correct' edge $\mbox{diam}(G)$ times in a row. This likelihood is, in general, extremely small and we are unlikely to succeed in the first attempt and it remains to understand how many such attempts are needed. The process is stochastically dominated by a simple geometric random variable with likelihood $p$.  Thus
  $$ \mathbb{E}\left( \mbox{number of attempts} \right) = \sum_{k=1}^{\infty} k (1-p)^{k-1} p = \frac{1}{p}.$$
Each such attempt is comprised of a random walk of length $\mbox{diam}(G)$ and thus the total number of steps is bounded from above by 
  $$  \mbox{hit}(G,v_1)  \leq \mbox{diam}(G) \cdot \Delta^{{\tiny \mbox{diam}(G)}}.$$
\end{proof}

\section*{Acknowledgements}
R.R.L. is supported by NIH/NIGMS (grant number R01GM136780) and AFOSR (grant number FA9550-21-1-0317). S.S. is supported by the NSF (grant number DMS-2123224) and the Alfred P. Sloan Foundation.

\end{document}